\documentclass[12pt]{amsart}

\usepackage{amssymb}
\usepackage{amscd}

\title{Galois points and rational functions with small value sets} 
\author{Satoru Fukasawa}

\address{Department of Mathematical Sciences, Faculty of Science, Yamagata University, Kojirakawa-machi 1-4-12, Yamagata 990-8560, Japan} 
\email{s.fukasawa@sci.kj.yamagata-u.ac.jp}

\thanks{The author was partially supported by JSPS KAKENHI Grant Number JP19K03438.}

\newtheorem{theorem}{Theorem}

\newtheorem{corollary}{Corollary} 
\newtheorem{fact}{Fact}

\theoremstyle{definition}
\newtheorem{remark}{Remark}

\begin{document}
\begin{abstract} 
This paper presents a connection between Galois points and rational functions over a finite field with small value sets. 
This paper proves that the defining polynomial of any plane curve admitting two Galois points is an irreducible component of a polynomial obtained as a relation of two rational functions. 
A recent result of Bartoli, Borges, and Quoos implies that one of these rational functions over a finite field has a very small value set, under the assumption that Galois groups of two Galois points generate the semidrect product. 
When two Galois points are external, this paper proves that the defining polynomial is an irreducible component of a polynomial with separated variables. 
This connects the study of Galois points to that of polynomials with small value sets.   
\end{abstract}

\maketitle 

\section{Introduction} 
The purpose of this paper is to present a connection between Galois points and rational functions over a finite field with small value sets. 
 
Let $C \subset \mathbb{P}^2$ be an irreducible plane curve of degree $d >1$ over an algebraically closed field $k$ of characteristic $p \ge 0$ and let $k(C)$ be its function field. 
Taking a point $P \in \mathbb{P}^2$, we consider the projection $\pi_P: C \dashrightarrow \mathbb{P}^1$ from $P$. 
A point $P \in \mathbb{P}^2$ is called a {\it Galois point} if the field extension $k(C)/\pi_P^*k(\mathbb{P}^1)$ of function fields induced by $\pi_P$ is a Galois extension (\cite{fukasawa1, miura-yoshihara, yoshihara}).  
The associated Galois group is denoted by $G_P$. 
Many results on Galois points have been obtained; however, there are many open problems (see \cite{fukasawa1, open}). 

The present author and Speziali investigated plane curves with two outer Galois points $P_1, P_2 \in \mathbb{P}^2 \setminus C$ such that $\langle G_{P_1}, G_{P_2} \rangle=G_{P_1} \rtimes G_{P_2}$ (\cite{fukasawa-speziali}), and the present author studied plane curves admitting an inner Galois point $P_1 \in C \setminus {\rm Sing}(C)$ and an outer Galois point $P_2 \in \mathbb{P}^2 \setminus C$ such that $\langle G_{P_1}, G_{P_2} \rangle=G_{P_1} \rtimes G_{P_2}$ or $G_{P_1} \ltimes G_{P_2}$ (\cite{fukasawa2}). 
In a more general situation, this paper proves the following. 

\begin{theorem} \label{main} 
Let $C \subset \mathbb{P}^2$ be defined over a finite field $\mathbb{F}_q$ of $q$ elements.  
Assume that $C$ is irreducible over the algebraic closure $\overline{\mathbb{F}}_q$ of $\mathbb{F}_q$. 
Let $P_1=(1:0:0), P_2=(0:1:0) \in \mathbb{P}^2$. 
If $P_1$ and $P_2$ are Galois points such that all automorphisms in $G_{P_1} \cup G_{P_2}$ are defined over $\mathbb{F}_q$, and $|\langle G_{P_1}, G_{P_2} \rangle| <\infty$, then the following holds. 
\begin{itemize}
\item[(I)] There exist polynomials $f_1, g_1, f_2, g_2 \in \mathbb{F}_q[x]$ such that 
\begin{itemize}
\item[(a)] $f_i$ and $g_i$ are relatively prime for $i=1, 2$, 
\item[(b)] $\max\{\deg f_i, \deg g_i\} = |\langle G_{P_1}, G_{P_2} \rangle|/|G_{P_j}|$ for $i, j$ with $\{i, j\}=\{1, 2\}$, \item[(c)] the defining polynomial of $C$ is an irreducible component of 
$$f_1(x)g_2(y)-g_1(x)f_2(y)$$ over $\mathbb{F}_q$. 
\end{itemize}
\end{itemize} 
Hereafter, we assume the existence of polynomials $f_1, g_1, f_2, g_2 \in \mathbb{F}_q[x]$ satisfying conditions (a), (b) and (c) in (I). 
Then the following hold. 
\begin{itemize}
\item[(II)]  If $|\langle G_{P_1}, G_{P_2} \rangle|=|G_{P_1}| \times |G_{P_2}|$, then the curve $C$ is defined by $$f_1(x)g_2(y)-g_1(x)f_2(y)=0. $$ 
The converse holds, if $\deg f_1 \ne \deg g_1$ or $\deg f_2 \ne \deg g_2$. 
\item[(III)] $\langle G_{P_1}, G_{P_2} \rangle=G_{P_1} \rtimes G_{P_2}$ if and only if $\mathbb{F}_q(y)/\mathbb{F}_q(h_2(y))$ is a Galois extension for $h_2(y)=f_2(y)/g_2(y)$. 
\item[(IV)] Assume that $P_1, P_2 \in \mathbb{P}^2 \setminus C$.  
Then we can take $g_1(x)=g_2(x)=1$, namely, a defining polynomial of $C$ is an irreducible component of $f_1(x)-f_2(y)$ over $\mathbb{F}_q$. 
In this case, $|\langle G_{P_1}, G_{P_2} \rangle|=d^2$ if and only if $f_1(x)-f_2(y)$ is a defining polynomial. 
\end{itemize}
\end{theorem}

\begin{remark}
\begin{itemize} 
\item[(a)] Theorem \ref{main} holds for any perfect field $k_0$, by replacing $\mathbb{F}_q$ by $k_0$. 
\item[(b)] In assertion (II), we always take $\deg f_1 \ne \deg g_1$, since if $\deg f_1=\deg g_1$, then $f_1/g_1=\alpha+f_{11}/g_1$ and $\mathbb{F}_q(f_1/g_1)=\mathbb{F}_q(f_{11}/g_1)$ for some $\alpha \in \mathbb{F}_q$ and $f_{11} \in \mathbb{F}_q[x]$ with $\deg f_{11} < \deg g_1$.  
\item[(c)] Galois points are defined over algebraically closed fields. 
Theorem \ref{main} suggests that it is good to define {\it a Galois point $P$ over a finite field $\mathbb{F}_q$} as an $\mathbb{F}_q$-rational point of $\mathbb{P}^2$ such that the extension $\mathbb{F}_q(C)/\mathbb{F}_q(L_1/L_2)$ is Galois, where $L_1, L_2 \in \mathbb{F}_q[X, Y, Z]$ are linearly independent homogeneous polynomials of degree one defining $P$.   
\end{itemize}
\end{remark}

{\it What are these rational functions $f_1/g_1$ and $f_2/g_2$?}
In a recent paper \cite{bbq}, Bartoli, Borges, and Quoos investigated the value sets of rational functions $h(x) \in \mathbb{F}_q(x)$, and obtained the following theorem.  

\begin{fact}[Bartoli, Borges, and Quoos] \label{value set} 
Let $f(x), g(x) \in \mathbb{F}_q[x]$ be relatively prime. 
If a rational function $h(x)=f(x)/g(x) \in \mathbb{F}_q(x)$ is such that $\mathbb{F}_q(x)/\mathbb{F}_q(h(x))$ is a Galois extension, then either 
$$ \# V_h=\left\lceil \frac{q+1}{\deg h} \right\rceil \ \mbox{ or } \ \# V_h=\left\lceil \frac{q+1}{\deg h} \right\rceil +1, $$
where $V_h=\{h(\alpha) \ | \ \alpha \in \mathbb{P}^1(\mathbb{F}_q)\} \subset \mathbb{P}^1(\mathbb{F}_q)$ and $\deg h=\max\{\deg f, \deg g\}$. 
\end{fact} 

Theorem \ref{main} and Fact \ref{value set} imply that the rational function $h_2(y)$ as in Theorem \ref{main} (III) has a very small value set. 
More precisely: 

\begin{corollary} 
Let $f_2(x), g_2(x) \in \mathbb{F}_q[x]$ be as in Theorem \ref{main} and let $h_2(x)=f_2(x)/g_2(x)$. 
If $\langle G_{P_1}, G_{P_2} \rangle = G_{P_1} \rtimes G_{P_2}$, then either
$$\# V_{h_2}=\left\lceil \frac{q+1}{\deg h_2} \right\rceil \ \mbox{ or } \ \# V_{h_2}=\left\lceil \frac{q+1}{\deg h_2} \right\rceil +1. $$
\end{corollary}

Theorem \ref{main} (IV) connects the study of Galois points to that of polynomials over finite fields. 
Borges \cite{borges} established a connection between minimal value set polynomials (\cite{clms, mills}) and Frobenius nonclassical curves (\cite{hefez-voloch, stohr-voloch}). 
A theorem \cite[Corollary 3.5]{borges} of Borges implies the following. 

\begin{corollary} \label{Frobenius nonclassical}
Assume that $P_1, P_2 \in \mathbb{P}^2 \setminus C$. 
Let $f_1(x), f_2(x) \in \mathbb{F}_q[x]$ be polynomials as in Theorem \ref{main} and let $V_{f_1}', V_{f_2}'$ be their value sets, that is, $V_{f_i}'=\{f_i(\alpha) \ | \ \alpha \in \mathbb{F}_q\}$ for $i=1, 2$. 
If $f_1, f_2$ are minimal value set polynomials such that $V_{f_1}'=V_{f_2}'$ and, $|V_{f_1}'|>2$ or $|V_{f_1}'|=2=p$, then $C$ is $q$-Frobenius nonclassical. 
\end{corollary}

A typical example of curves satisfying the assumptions in Corollary \ref{Frobenius nonclassical} is the Fermat curve 
$$ x^{\frac{q-1}{q'-1}}+y^{\frac{q-1}{q'-1}}+1=0, $$
where $\mathbb{F}_{q'} \subset \mathbb{F}_q$. 
Points $(1:0:0)$, $(0:1:0)$ are outer Galois points (\cite{fukasawa1, miura-yoshihara, yoshihara}), and polynomials $x^{\frac{q-1}{q'-1}}$ and $-y^{\frac{q-1}{q'-1}}-1$ have the same minimal value set $\mathbb{F}_{q'}$ (\cite{borges}). 
Another example is found in \cite[Theorem 2]{borges-fukasawa}. 

\begin{remark}
Assume that $f(x) \in \mathbb{F}_q[x]$ and a field extension $\mathbb{F}_q(x)/\mathbb{F}_q(f(x))$ is Galois. 
It is well known that a place at infinity is a total ramification point and there exist at most two short orbits. 
It can be confirmed that $f(x)$ is a minimal value set polynomial, by a method similar to the proof of Fact \ref{value set} (see \cite[Proof of Theorem 2.1]{bbq}). 
\end{remark}

\section{Proofs} 

\begin{proof}[Proof of Theorem \ref{main}] 
Assume that points $P_1=(1:0:0), P_2=(0:1:0) \in \mathbb{P}^2$ are Galois points, and that the group $G:=\langle G_{P_1}, G_{P_2} \rangle$ is of finite order. 
The projections $\pi_{P_1}$ and $\pi_{P_2}$ from points $P_1$ and $P_2$ are represented by 
$$ \pi_{P_1}(x, y)=y \ \mbox{ and } \ \pi_{P_2}(x, y)=x $$
respectively. 
Since all elements of $G_{P_1} \cup G_{P_2}$ are defined over $\mathbb{F}_q$, 
it follows that $\mathbb{F}_q(C)^{G_{P_1}}=\mathbb{F}_q(y)$ and $\mathbb{F}_q(C)^{G_{P_2}}=\mathbb{F}_q(x)$. 
Since $|G| < \infty$, by L\"{u}roth's theorem, there exists a function $t \in \mathbb{F}_q(C)^{G}$ such that $\mathbb{F}_q(t)=\mathbb{F}_q(C)^{G}$. 
Since $\mathbb{F}_q(t) \subset \mathbb{F}_q(y)$ and $\mathbb{F}_q(t) \subset \mathbb{F}_q(x)$, there exist polynomials $f_2(y), g_2(y) \in \mathbb{F}_q[y]$ and $f_1(x), g_1(x) \in \mathbb{F}_q[x]$ such that    
$$t=f_2(y)/g_2(y) \ \mbox{ and } \ t=f_1(x)/g_1(x). $$  
We can assume that polynomials $f_i(x)$ and $g_i(x)$ are relatively prime for $i=1, 2$. 
Let $h_i(x)=f_i(x)/g_i(x)$ for $i=1, 2$. 
Since
$$\mathbb{F}_q(y)/\mathbb{F}_q(h_2(y))=\mathbb{F}_q(C)^{G_{P_1}}/\mathbb{F}_q(C)^{G}, \ \mathbb{F}_q(x)/\mathbb{F}_q(h_1(x))=\mathbb{F}_q(C)^{G_{P_2}}/\mathbb{F}_q(C)^{G}, $$ 
it follows that 
$\max\{\deg f_2, \deg g_2\}=|G|/|G_{P_1}|$ and $\max\{\deg f_1, \deg g_1\}=|G|/|G_{P_2}|$. 
Since $f_1(x)/g_1(x)=t=f_2(y)/g_2(y)$ in $\mathbb{F}_q(C)$, it follows that 
$$ f(x, y):=f_1(x)g_2(y)-g_1(x)f_2(y)=0 $$
in $\mathbb{F}_q(C)$.
Assertion (I) follows.  

Assume that $|G|=|G_{P_1}| \times |G_{P_2}|$. 
Note that
$$ |G_{P_1}|=|G|/|G_{P_2}|=\max\{\deg f_1, \deg g_1\}.$$
Since 
$$
\deg_x f(x, y) \le \max\{\deg f_1(x), \deg g_1(x)\}= |G_{P_1}|=\deg \pi_{P_1}, 
$$
it follows that $\deg_x f(x, y)=\deg \pi_{P_1}$ and $f(x, y)$ is a minimal polynomial of $x$ over $\overline{\mathbb{F}}_q(y)$. 
This implies that $f(x, y)$ is irreducible as an element of $\overline{\mathbb{F}}_q(y)[x]$. 
Therefore, $f(x, y)$ is irreducible in $\overline{\mathbb{F}}_q[x, y]$.  

Assume that $f(x, y)$ is a defining polynomial of $C$ and $\deg f_1 \ne \deg g_1$. 
Then $\max\{\deg f_1, \deg g_1\} = \deg \pi_{P_1}=|G_{P_1}|$.
Since $|G|/|G_{P_2}|=\max\{\deg f_1, \deg g_1\}$, it follows that 
$$ |G|=|G_{P_1}| \times |G_{P_2}|. $$
Assertion (II) follows. 

$G=G_{P_1} \rtimes G_{P_2}$ if and only if $G_{P_1}$ is a normal subgroup of $G$. 
Assertion (III) follows, by Galois theory. 

Assume that $P_1, P_2 \in \mathbb{P}^2 \setminus C$. 
Let $\varphi_i: \mathbb{P}^1 \rightarrow \mathbb{P}^1$ be the morphism corresponding to $\overline{\mathbb{F}}_q(C)^{G_{P_i}}/\overline{\mathbb{F}}_q(C)^{G}$ for $i=1,2$.  
Let $Q$ be a place of $\overline{\mathbb{F}}_q(C)$ coming from $C \cap \overline{P_1P_2}$, where $\overline{P_1P_2}$ be a line passing through $P_1$ and $P_2$. 
Since the fiber of $\varphi_i(\pi_{P_i}(Q))$ for the covering $\varphi_i \circ \pi_{P_i}$ coincides with the orbit $G \cdot Q$ (see \cite[III.7.1]{stichtenoth}), it follows that 
$$ \varphi_i^{-1} (\varphi_i(\pi_{P_i}(Q)))=\pi_{P_i}(G \cdot Q),  $$
for $i=1, 2$. 
Since $P_1, P_2 \in \mathbb{P}^2 \setminus C$, it follows that 
$$\pi_{P_i}(G \cdot Q)=\{\pi_{P_i}(Q)\}, $$ 
and that $\varphi_i$ is totally ramified at $\pi_{P_i}(Q)$, for $i=1, 2$. 
We take a system $(Y:Z)$ of coordinates on $\pi_{P_1}(C) \cong \mathbb{P}^1$ (resp. a system $(X:Z)$ of coordinates on $\pi_{P_2}(C) \cong \mathbb{P}^1$) such that $\pi_{P_1}(Q)=(1:0)$ (resp. $\pi_{P_2}(Q)=(1:0)$). 
Note that 
$$\varphi_1(\pi_{P_1}(Q))=\varphi_2(\pi_{P_2}(Q)).$$ 
We take a system $(t:1)$ of coordinates on $\varphi_1(\pi_{P_1}(C))=\varphi_2(\pi_{P_2}(C)) \cong \mathbb{P}^1$ such that 
$$\varphi_1(\pi_{P_1}(Q))=(1:0)=\varphi_2(\pi_{P_2}(Q)).$$ 
Since $\varphi_1$ (resp. $\varphi_2$) is totally ramified at $(1:0)$ and $\varphi_1(1:0)=(1:0)$ (resp. $\varphi_2(1:0)=(1:0)$), it follows that $\varphi_1(y:1)=(f_2(y):1)$ (resp. $\varphi_2(x:1)=f_1(x)$) for some polynomial $f_2(y) \in \mathbb{F}_q[y]$ (resp. $f_1(x) \in \mathbb{F}_q[x]$). 
Since $f_2(y)=t=f_1(x)$ in $\mathbb{F}_q(C)$, the former assertion of (IV) follows. 
The latter assertion of (IV) comes from assertion (II). 
\end{proof}

Corollary \ref{Frobenius nonclassical} is derived from Borges's theorem \cite[Corollary 3.5]{borges}. 
In \cite[Theorem 3.4, Corollary 3.5]{borges}, it is assumed that all irreducible components of $f(x)-g(y)$ are defined over $\mathbb{F}_q$. 
Therefore, we confirm that the reasoning in Borges's paper \cite{borges} can be applied to our case, more precisely, we confirm that any component of $f_1(x)-f_2(y)$ defined over $\mathbb{F}_q$ is $q$-Frobenius nonclassical, under the assumption on $f_1, f_2$ as in Corollary \ref{Frobenius nonclassical}. 

\begin{proof}[Proof of Corollary \ref{Frobenius nonclassical}]
Let $P_1, P_2 \in \mathbb{P}^2 \setminus C$, and let $f_1, f_2 \in \mathbb{F}_q[x]$ be polynomials as in Theorem \ref{main}. 
Assume that $f_1, f_2$ are minimal value set polynomials such that $V_{f_1}'=V_{f_2}'$ and, $|V_{f_1}'|>2$ or $|V_{f_1}'|=2=p$. 
By \cite[Theorem 2.2]{borges}, there exist $\theta_i \in \mathbb{F}_q^*$ and a monic polynomial 
$T_i=\prod_{\gamma \in V_{f_i}'} (x-\gamma) \in \mathbb{F}_q[x]$ such that 
$$ T_i(f_i)=\theta_i(x^q-x)f_{i, x} $$
for $i=1, 2$, where $f_{i, x}$ is the formal derivative of $f_i$ by $x$. 
Since $V_{f_1}'=V_{f_2}'$, it follows that $T_1=T_2$. 
By \cite[Lemma 2.4(ii)]{borges}, $\theta_1=\theta_2$. 
Since $X-Y$ divides $T_1(X)-T_1(Y)$, it follows that $f(x, y)=f_1(x)-f_2(y)$ divides 
$$ (x^q-x)f_x+(y^q-y)f_y=(x^q-x)f_{1, x}-(y^q-y)f_{2, y}. $$
Since the defining polynomial $f_0$ of $C$ is an irreducible component of $f(x, y)$, it follows from \cite[Lemmas 3.2]{borges} and \cite[Lemma 3.3 (i) $\Rightarrow$ (ii)]{borges} that $f_0$ divides
$$ (x^q-x)f_{0, x}+(y^q-y)f_{0, y}, $$
that is, $C$ is $q$-Frobenius nonclassical. 
\end{proof}

\begin{center} {\bf Acknowledgements} \end{center} 
The author is grateful to Doctor Kazuki Higashine for helpful discussions.


\begin{thebibliography}{20} 
\bibitem{bbq} D. Bartoli, H. Borges and L. Quoos, Rational functions with small value set, J. Algebra {\bf 565} (2021), 675--690. 
\bibitem{borges} H. Borges, Frobenius nonclassical components of curves with separated variables, J. Number Theory {\bf 159} (2016), 402--425. 
\bibitem{borges-fukasawa} H. Borges and S. Fukasawa, An elementary $p$-cover of the Hermitian curve with many automorphisms, preprint, arXiv:2009.04716. 
\bibitem{clms} L. Carlitz, D. J. Lewis, W. H. Mills and E. G. Straus, Polynomials over finite fields with minimal value sets, Mathematika {\bf 8} (1961), 121--130. 
\bibitem{fukasawa1} S. Fukasawa, Galois points for a plane curve in arbitrary characteristic, Geom. Dedicata {\bf 139} (2009), 211--218.  
\bibitem{fukasawa2} S. Fukasawa, Algebraic curves admitting inner and outer Galois points, preprint, arXiv:2010.00815. 
\bibitem{fukasawa-speziali} S. Fukasawa and P. Speziali, Plane curves possessing two outer Galois points, preprint, arXiv:1801.03198. 
\bibitem{hefez-voloch} A. Hefez and J. F. Voloch, Frobenius non classical curves, Arch. Math. {\bf 54} (1990), 263--273. 
\bibitem{mills} W. H. Mills, Polynomials with minimal value sets, Pacific J. Math. {\bf 14} (1964), 225--241. 
\bibitem{miura-yoshihara} K. Miura and H. Yoshihara, Field theory for function fields of plane quartic curves, J. Algebra {\bf 226} (2000), 283--294. 
\bibitem{stichtenoth} H. Stichtenoth, Algebraic Function Fields and Codes, Universitext, Springer-Verlag, Berlin, 1993. 
\bibitem{stohr-voloch} K.-O. St\"{o}hr and J. F. Voloch, Weierstrass points and curves over finite fields, Proc. Lond. Math. Soc. (3) {\bf 52} (1986), 1--19. 
\bibitem{yoshihara} H. Yoshihara, Function field theory of plane curves by dual curves, J. Algebra {\bf 239} (2001), 340--355. 
\bibitem{open} H. Yoshihara and S. Fukasawa, List of problems, available at: \\ https://sites.google.com/sci.kj.yamagata-u.ac.jp/fukasawa-lab/open-questions-english
\end{thebibliography}
\end{document}